\documentclass[letterpaper,10pt,reqno]{amsart}
\usepackage{indentfirst} 
\usepackage{amssymb}
\usepackage{mathtools} 
\usepackage{mathabx} 
\usepackage{amsthm}
\usepackage{thmtools}
\usepackage{enumitem} 
\usepackage[backref=page,colorlinks=true]{hyperref} 
\usepackage[usenames,dvipsnames]{xcolor} 
\usepackage{ifthen}
\usepackage{tikz}
\usetikzlibrary{decorations.pathreplacing}
\usepackage{tikz-cd}
  
\renewcommand*{\backref}[1]{}
\renewcommand*{\backrefalt}[4]{\tiny
  \ifcase #1 (\textbf{NOT CITED.})%
  \or    (Cited on page~#2.)%
  \else   (Cited on pages~#2.)%
  \fi}

\makeatletter

\def\MRbibitem{\@ifnextchar[\my@lbibitem\my@bibitem}

\def\mybiblabel#1#2{\@biblabel{{\hyperref{http://www.ams.org/mathscinet-getitem?mr=#1}{}{}{#2}}}}

\def\myhyperanchor#1{\Hy@raisedlink{\hyper@anchorstart{cite.#1}\hyper@anchorend}}

\def\my@lbibitem[#1]#2#3#4\par{%
  \item[\mybiblabel{#2}{#1}\myhyperanchor{#3}\hfill]#4%
  \@ifundefined{ifbackrefparscan}{}{\BR@backref{#3}}%
  \if@filesw{\let\protect\noexpand\immediate
    \write\@auxout{\string\bibcite{#3}{#1}}}\fi\ignorespaces%
}

\def\my@bibitem#1#2#3\par{%
  \refstepcounter\@listctr
  \item[\mybiblabel{#1}{\the\value\@listctr}\myhyperanchor{#2}\hfill]#3%
  \@ifundefined{ifbackrefparscan}{}{\BR@backref{#2}}%
  \if@filesw\immediate\write\@auxout
    {\string\bibcite{#2}{\the\value\@listctr}}\fi\ignorespaces%
}

\makeatother

\DeclareFontFamily{U} {MnSymbolA}{}
\DeclareFontShape{U}{MnSymbolA}{m}{n}{
   <-6> MnSymbolA5
   <6-7> MnSymbolA6
   <7-8> MnSymbolA7
   <8-9> MnSymbolA8
   <9-10> MnSymbolA9
   <10-12> MnSymbolA10
   <12-> MnSymbolA12}{}
\DeclareFontShape{U}{MnSymbolA}{b}{n}{
   <-6> MnSymbolA-Bold5
   <6-7> MnSymbolA-Bold6
   <7-8> MnSymbolA-Bold7
   <8-9> MnSymbolA-Bold8
   <9-10> MnSymbolA-Bold9
   <10-12> MnSymbolA-Bold10
   <12-> MnSymbolA-Bold12}{}
\DeclareSymbolFont{MnSyA} {U} {MnSymbolA}{m}{n}
 \DeclareFontFamily{U} {MnSymbolC}{}
\DeclareFontShape{U}{MnSymbolC}{m}{n}{
  <-6> MnSymbolC5
  <6-7> MnSymbolC6
  <7-8> MnSymbolC7
  <8-9> MnSymbolC8
  <9-10> MnSymbolC9
  <10-12> MnSymbolC10
  <12-> MnSymbolC12}{}
\DeclareFontShape{U}{MnSymbolC}{b}{n}{
  <-6> MnSymbolC-Bold5
  <6-7> MnSymbolC-Bold6
  <7-8> MnSymbolC-Bold7
  <8-9> MnSymbolC-Bold8
  <9-10> MnSymbolC-Bold9
  <10-12> MnSymbolC-Bold10
  <12-> MnSymbolC-Bold12}{}
\DeclareSymbolFont{MnSyC} {U} {MnSymbolC}{m}{n}

\DeclareMathSymbol{\top}{\mathord}{MnSyA}{219} 
\DeclareMathSymbol{\plus}{\mathord}{MnSyC}{20} 

\declaretheorem[numberwithin=section]{theorem}
\declaretheorem[sibling=theorem]{lemma}
\declaretheorem[sibling=theorem]{corollary}
\declaretheorem[sibling=theorem]{proposition}
\declaretheorem[sibling=theorem,style=definition]{definition}
\declaretheorem[sibling=theorem,style=remark]{example}

\declaretheorem[sibling=theorem,style=remark]{remark}

\hypersetup{bookmarksdepth = 3} 
\numberwithin{equation}{section}     

\setlist[enumerate,1]{label={\upshape(\alph*)},ref=\alph*}
\setlist[enumerate,2]{label={\upshape(\arabic*)},ref=\arabic*}

\newcommand{\R}{\mathbb{R}}
\newcommand{\Z}{\mathbb{Z}}
\newcommand{\N}{\mathbb{N}}

\renewcommand{\H}{\mathbb{H}}

\newcommand{\cK}{\mathcal{K}}

\newcommand{\cP}{\mathcal{P}}
\newcommand{\cU}{\mathcal{U}}
\newcommand{\cV}{\mathcal{V}}

\newcommand{\Rnon}{\mathbb{R}_{\plus}}    
\newcommand{\Rpos}{\mathbb{R}_{\plus\plus}} 
\newcommand{\Mat}[2][]{\ifthenelse{\equal{#1}{}}{\R^{{#2}\times{#2}}}{\R^{{#1}\times{#2}}}}
\newcommand{\Man}[2][]{\ifthenelse{\equal{#1}{}}{\Rnon^{{#2}\times{#2}}}{\Rnon^{{#1}\times{#2}}}}
\newcommand{\Map}[2][]{\ifthenelse{\equal{#1}{}}{\Rpos^{{#2}\times{#2}}}{\Rpos^{{#1}\times{#2}}}}

\renewcommand{\epsilon}{\varepsilon}
\renewcommand{\phi}{\varphi}
\renewcommand{\setminus}{\smallsetminus}

\begin{document}

\title{Pressure, Poincar\'e series and box dimension of the boundary}
\date{\today}



\author[G.~Iommi]{Godofredo Iommi}
\address{Facultad de Matem\'aticas,
Pontificia Universidad Cat\'olica de Chile (PUC), Avenida Vicu\~na Mackenna 4860, Santiago, Chile}
\email{giommi@mat.puc.cl}
\urladdr{\url{http://http://www.mat.uc.cl/~giommi/}}

\author[A.~Velozo]{Anibal Velozo}  \address{Department of Mathematics, Yale University, New Haven, CT 06511, USA.}
\email{anibal.velozo@yale.edu }
\urladdr{\url{https://gauss.math.yale.edu/~av578/}}

\maketitle

\begin{abstract}
In this note we prove two related results. First, we show that for certain Markov interval maps with infinitely many branches the upper box dimension of the boundary can be read from the pressure of the geometric potential. Secondly, we prove that the box dimension  of the set of iterates of a point in $\partial \H^n$ with respect to a parabolic subgroup of isometries equals the critical exponent of the Poincar\'e series of the associated group. This establishes a relationship between the entropy at infinity and dimension theory.
\end{abstract}
\vspace{1cm}

\section{Introduction}

The application of thermodynamic formalism to the dimension theory of dynamical systems dates to the work of Bowen \cite{bo}, who related the Hausdorff dimension of a dynamically defined set to the root of a certain pressure function. More precisely, he proved that the Hausdorff dimension of the limit set of a quasi-Fuchsian group can be recovered from the pressure of a suitably chosen potential (see \cite[Lemma 10]{bo}).  Bowen's beautiful result highlights the relation between thermodynamic formalism and the dimension theory of limit sets. In the present article we prove, for two classes of dynamical systems defined on non-compact spaces, a new relation between these two theories.


We first consider a class of Markov interval maps with countably many branches called Expanding-Markov-Renyi interval maps (EMR). These are maps of the form $T:\bigcup_{n=1}^{\infty}I_n \to [0,1]$, where $I_n=[a_n, b_n]$, and $(I_n)_{n=1}^\infty$ is a collection of closed intervals contained in $[0,1]$ with disjoint interiors. Mauldin and Urba\'nski \cite{mu} proved that the Hausdorff dimension of the non-compact repeller $\Lambda$ of an EMR map $T$ is essentially the root of an equation involving the pressure of the geometric potential (see Proposition \ref{prop:bow}). They proved that
\begin{equation*}
\dim_H(\Lambda)= \inf \left\{t \in \R	: P(t) \leq 0	\right\},		
\end{equation*}
where $\dim_H$ denotes the Hausdorff dimension and $P(t)$ the pressure function (for precise definitions see Section \ref{s:2}). This generalizes Bowen's result to non-compact settings (see \cite[Chapter 5]{fa2} for the result in the setting of expanding Markov maps with finitely many branches). 

In this context, the pressure function $P(t)$ has a critical value $s_{\infty} \in (0, \infty)$ such that if $t<s_{\infty}$ then $P(t)= \infty$ and if $t>s_{\infty}$ then the pressure is finite. The relation between the critical value $s_{\infty}$ and the multifractal analysis theory has been studied in \cite{fjlr, ij}. For example, in \cite[Section 7]{fjlr} it is proved that the value $s_{\infty}$ is a lower bound for the Hausdorff dimension of sets of numbers having a prescribed frequency of digits in their continued fraction expansion. It was shown in \cite[Theorem 5.1]{ij} that the behaviour of the pressure at $s_{\infty}$ determines regularity properties of the multifractal spectrum of Lyapunov exponents. Moreover, in \cite[Theorem 7.1]{ij} it is proved that the Hausdorff dimension of the set of points having infinite Lyapunov exponent is precisely $s_{\infty}$. 

 In this article we prove a new relation between $s_{\infty}$ and the dimension theory of $T$. A novelty being that instead of Hausdorff, we consider box dimension. We denote by  $\overline{\dim}_B$ and $\dim_B$ the upper box dimension and the box dimension, respectively. One of the main results in this paper is the following (see Theorem \ref{main}). 

\begin{theorem} 
Let $T$ be an EMR map. Then
\begin{equation*}
s_{\infty}\le\overline{\dim}_B \left(\bigcup_{n=1}^{\infty} \left\{a_n, b_n	\right\} \right).
\end{equation*}
If the box dimension of $\bigcup_{n=1}^{\infty} \left\{a_n, b_n	\right\}$ exists, then   $$s_\infty=\dim_B\bigcup_{n=1}^{\infty} \left\{a_n, b_n	\right\}.$$
\end{theorem} 
Depending on the properties of the pressure function at the critical value we are able to derive certain stability results. For example, we prove existence of measures of maximal dimension for perturbations of the original map (see Corollary \ref{cor:pert}).



Bowen's work on the dimension theory of limit sets has been further developed by Sullivan. In a landmark paper \cite{s} Sullivan proved that the Hausdorff dimension of the limit set of a geometrically finite Kleinian group $\Gamma\leqslant\text{Iso}(\H^n)$  coincides with the critical exponent of $\Gamma$, which is denoted by $\delta_\Gamma$ (see \cite[Theorem 25]{s}). For precise definitions we refer the reader to Section \ref{sec:geo}. A major generalization of Sullivan's result was obtained by Bishop and Jones in \cite{bj}, where they proved that the critical exponent $\delta_\Gamma$ is equal to the Hausdorff dimension of the radial limit set of $\Gamma$ for an arbitrary Kleinian group (see \cite[Theorem 2.1]{bj}). The critical exponent of $\Gamma$ has yet another dynamical interpretation: it is equal to the topological entropy of the geodesic flow on $T^1(\H^n/\Gamma)$ (see \cite[Theorem 1]{op}). 
In other words we have that $\delta_\Gamma=\sup_{\mu}h_\mu(g_1)$, where the supremum runs over the space of invariant probability measures of the geodesic flow and $g_1$ is the time-one map of the geodesic flow. 

In recent works a new dynamical invariant has been studied, the entropy at infinity of the geodesic flow (see \cite{irv}, \cite{rv}, \cite{v2}).  For completeness we recall its definition here. We say that a sequence of measures $(\mu_n)_n$ converges in the vague topology to $\mu$ if $\lim_{n\to\infty}\int fd\mu_n=\int fd\mu$, for every continuous function of compact support $f$. The entropy at infinity of the geodesic flow on $\H^n/\Gamma$ is defined by 
$$h_\infty(\Gamma):=\sup_{(\mu_n)_n\to 0}\limsup_{n\to\infty}h_{\mu_n}(g_1),$$
where the supremum runs over sequences of invariant probability measures $(\mu_n)_n$  that converge in the vague topology to the zero measure. The entropy at infinity of the geodesic flow is strongly related to the upper-semicontinuity of the entropy map and it is an important tool in the study of the thermodynamic formalism of the geodesic flow (see \cite{irv}, \cite{rv}, \cite{v1}, \cite{v2}). 

From now on assume that $\Gamma\leqslant \text{Iso}(\H^n)$ is geometrically finite. In this case the topological entropy of the geodesic flow on $\H^n/\Gamma$ is equal to the Hausdorff dimension of the limit set of $\Gamma$. Because of this relation it is natural to ask if there is also a relation between the entropy at infinity and dimension theory of the boundary. One of the main goals of this paper is to prove that for geometrically finite manifolds this is indeed the case. It is important to mention that under the geometrically finite assumption we have that  $h_\infty(\Gamma)=\sup_\mathcal{P} \delta_\mathcal{P},$ where the supremum runs over the parabolic subgroups of $\Gamma$ (see \cite[Theorem 1.3]{rv}). In other words, the entropy at infinity is determined by the critical exponent of the parabolic subgroups of $\Gamma$. In Section \ref{sec:geo} we prove the following result.
\begin{theorem} Let $\mathcal{P}\leqslant\text{\normalfont Iso}(\H^n)$ be a parabolic subgroup and $\xi\in \partial \H^n$ a point which is not fixed by $\mathcal{P}$. Then 
$$\delta_\mathcal{P}=\dim_B\left(\bigcup_{p\in \mathcal{P}}p\xi\right),$$
where $\dim_B$ is the box dimension computed using the spherical metric on $\partial\H^n$. 
\end{theorem}
In Section \ref{sec:geo} we also obtain very similar results to those in Sub-section \ref{www} in the context of $Cat(-1)$ surfaces. 

It is interesting to note the relation between the ergodic theory of countable Markov shifts (or related dynamical systems such as the interval maps considered in this article) and the ergodic theory of geodesic flows on non-compact complete hyperbolic manifolds. In some cases  Markov partitions can be constructed for the geodesic flow and then the relation is rather explicit (see \cite{dp, irv}). 
Unfortunately, in general it is not known if there exists a symbolic coding for the geodesic flow. Despite this, in recent  years thermodynamic formalism  for geodesic flows has been extensively studied and  results have been proved in analogy to those obtained for countable Markov shifts  (see for example  \cite{pps, ps, rv, v1,v2}). On the other hand, the ergodic theory of countable Markov shifts has been studied mirroring results obtained for geodesic flows. For example, continuity properties of the entropy map and its relations with escape of mass were obtained in \cite{rv,v2} for the geodesic flow and later in \cite{itv1, itv2} for countable Markov shifts. The results in this paper are further evidence of the strong relation between these two classes of dynamical systems.\\

%

\noindent
{\bf Acknowledgment.} We thank Neil Dobbs and Amitesh Datta for many useful comments on this paper. G.I. was partially supported by CONICYT PIA ACT172001 and by Proyecto Fondecyt 1190194.

\section{Pressure and dimension}\label{s:2}

\subsection{Dimension theory}     \label{dim}
In this sub-section we recall some basic definitions and results from dimension theory that will be used in what follows. For a complete account on dimension theory we refer the reader to \cite{fa} and \cite{fa2}. 

Let $(M,d)$ be a metric space. The diameter of a set $B\subset M$ is denoted by $|B|$. A countable collection of subsets $\{U_i \}_{i\in \N}$ of $M$,   is called a $\delta$-cover of $F \subset M$ if $F\subset\bigcup_{i\in\N} U_i$, and $|U_i|$ is at most $\delta$ for every $i\in\N$. Letting $s>0$, we define
\[
\mathcal{H}^s(J) := \lim_{\delta \to 0}\inf \left\{ \sum_{i=1}^{\infty} |U_i|^s : \{U_i \}_i \text{ is a } \delta\text{-cover of } J \right\}.
\]
The \emph{Hausdorff dimension} of the set $J$ is defined by
\[
{\dim_H}(J) := \inf \left\{ s>0 : \mathcal{H}^s(J) =0 \right\}.
\]
The \emph{Hausdorff dimension} of a Borel measure $\mu$ is defined by
\[
{\dim_H}(\mu) := \inf \left\{ {\dim_H}(Z): \mu(Z)=1 \right\}.
\]
For a detailed discussion on the Hausdorff dimension see \cite[Chapter 2]{fa}.

An alternative definition of dimension that will be central to this work is  the \emph{box dimension}. Let $F \subset M$ and $N_{\delta}(F)$ be the smallest number of sets of diameter at most $\delta$ needed to cover $F$. The \emph{lower} and \emph{upper box dimensions} are defined by
\begin{equation*}
\underline{\dim_B}(F):= \liminf_{\delta \to 0} \frac{\log N_{\delta}(F)}{-\log \delta} \quad , \quad \overline{\dim_B}(F):= \limsup_{\delta \to 0} \frac{\log N_{\delta}(F)}{-\log \delta}
\end{equation*}
If the limits above are equal then we call this common value  the \emph{box dimension} of the set $F$,
\[ \dim_B(F):= \lim_{\delta \to 0} \frac{\log N_{\delta}(F)}{-\log \delta}. \]

\begin{remark}
It is useful to note that in the definition of box dimension it is possible to replace the number $N_{\delta}(F)$ by the largest number of disjoint balls of radius $\delta$ and  centers in $F$ (see \cite[p.41]{fa}).
\end{remark}
The following result will be of great importance in this paper (see \cite[Proposition 3.6 and 3.7]{fa2}).

\begin{proposition}
Let $I_n=[a_n , b_n ]$ be a sequence of intervals with disjoint interiors such that $[0,1]=\bigcup_{n=1}^\infty I_n$ and let $F= \bigcup_{n=1}^{\infty} \{ a_n, b_n \}$. Assume that $(b_n-a_n)_n$ is non-increasing.  Then
\begin{equation*}
\left(-\liminf_{n \to \infty} \frac{\log (b_n-a_n)}{\log n} \right)^{-1}	\leq \underline{\dim_B}(F) \leq \overline{\dim_B}(F) 
\leq \left(-\limsup_{n \to \infty} \frac{\log (b_n-a_n)}{\log n} 	\right)^{-1}.
\end{equation*}
\end{proposition}

\subsection{Markov shifts and Markov maps} \label{mar}
In this Sub-section we explain the symbolic structure we assume for the one dimensional maps we consider. 

The full-shift on the countable alphabet $\mathbb{N}$ is the pair $(\Sigma, \sigma)$ where
$\Sigma :=\N^\N$
and $\sigma: \Sigma \to \Sigma$ is the \emph{shift} map  defined by $\sigma(\omega_1, \omega_2, \cdots)=(\omega_2, \omega_3,\cdots)$.
We equip $\Sigma$ with the topology generated by the cylinders sets
\begin{equation*}
 C_{i_1 \cdots i_n}:=  \{ \omega \in \Sigma :\omega_j= i_j, \forall j \in \{1, \dots, n\} \}.
\end{equation*} 

Denote by  $I=[0,1]$ the unit interval. Let  $\{ I_i \}_{i\in\N}$ be a countable collection of closed intervals where  $\textrm{int}(I_i)\cap \textrm{int}(I_j)=\emptyset$, for $i,j\in\N$ with $i\neq j$, and $[a_i, b_i]:=I_i \subset I$ for every $i \in \mathbb{N}$. Let $T:\bigcup_{n=1}^{\infty}I_n \to I$ 
be a map. The \emph{repeller} of such a map is defined by
\[\Lambda:=\left\{x \in \bigcup_{i=1}^{\infty} I_i: T^n(x) \textrm{ is well defined for every } n \in \mathbb{N} \right\}.\]
Let $\mathcal{O}:= \bigcup_{k=0}^{\infty} T^{-k}(\bigcup_{i=0}^{\infty} \{a_i, b_i\})$. 
We say that $T$ is \emph{Markov} and it can be coded by a full-shift on a countable alphabet if  there exists a homemorphism  $\pi: \Sigma \to \Lambda \setminus \mathcal{O}$ such that $T \circ \pi =\pi  \circ  \sigma$. Denote by $I_{i_1 \dots i_n} \subset I$ the projection of the symbolic cylinder $C_{i_1 \dots i_n}$ by $\pi$.

\subsection{The class of maps}
The  class of EMR (expanding-Markov-Renyi) interval maps  was considered  by Pollicott and Weiss in \cite{pw} in their study of multifractal analysis.

\begin{definition} \label{def:emr}
 Let $\{ I_i \}_{i\in\N}$ be a countable collection of closed intervals where  $\textrm{int}(I_i)\cap \textrm{int}(I_j)=\emptyset$ for $i,j\in\N$ with $i\neq j$ and $[a_i, b_i]:=I_i \subset I$ for every $i \in \mathbb{N}$.
A map $T:\bigcup_{n=1}^{\infty}I_n \to I$ is an EMR map, if the following properties are satisfied
\begin{enumerate}
\item The only accumulation point for the end points of the intervals $[a_i, b_i]$ is $x=0$. 
\item The map is $C^2$ on $\bigcup_{i=1}^{\infty} \textrm{int }I_i$.
\item There exists $\xi >1$ and $N\in\N$ such that for every $x \in \bigcup_{i=1}^{\infty} I_i$ and $n\geq N$
we have $|(T^n)'(x)|>\xi^n$.
\item The map $T$ is Markov and it can be coded by a full-shift on a countable alphabet.
\item The map satisfies the Renyi condition, that is, there exists a positive number $K>0$
such that
\[ \sup_{n \in \N} \sup_{x,y,z \in I_n} \frac{|T''(x)|}{|T'(y)| |T'(z)|} \leq K. \]
\end{enumerate}
\end{definition}

%

\begin{example}\label{ex:g}
The Gauss map $G:(0,1] \to (0,1]$ defined by
\[G(x)= \frac{1}{x} -\Big[ \frac{1}{x} \Big],\]
where $[ \cdot]$ is the integer part, satisfies our assumptions. The Gauss map with restricted digits (that is the Gauss map with  branches erased so that there are still infinitely many branches left) is also a EMR map. 
\end{example}

%

The following is a fundamental property of EMR maps, see \cite[Chapter 7 Section 4]{cfs} or \cite[p.149]{pw}.
\begin{lemma} \label{lem:je}
There exists a positive constant $C>0$ such that for every $x \in \Lambda \setminus \mathcal{O}$ with $x \in I_{i_1 \dots i_n}$  we have
\begin{equation*}
\frac{1}{C} \leq \sup_{n \geq 0} \sup_{y \in I_{i_1 \dots i_n}} \left|	 \frac{(T^n)'(x)}{(T^n)'(y)}		\right| \leq C.
\end{equation*}
\end{lemma}

\subsection{Thermodynamic formalism and Hausdorff dimension} \label{sec:termo_haus}
Thermodynamic formalism is a set of tools brought to ergodic theory from statistical mechanics in the 1960s that allows for the choice of relevant invariant probability measures. It has surprising and interesting applications to the dimension theory of dynamical systems. The thermodynamic formalism of  EMR maps and regular potentials has been extensively  studied and it is fairly well understood (see for example \cite{ij,mu, pw, sa}). We now summarize some known results. 

\begin{definition}
The \emph{pressure} of $T$ at the point $t \in \R$ is defined by
\begin{equation*}
P(t)=\lim_{n \to \infty} \frac{1}{n} \log \sum_{T^n x=x}  \left(\prod_{i=0}^{n-1} |T'(T^i x)|^{-t} \right).
\end{equation*}
\end{definition}

It worth emphasizing that the pressure is usually defined over a large class of functions (which in analogy to statistical mechanics are called potentials). Adopting that point of view, the function $P(t)$ is equal to to the pressure of the potential $tf$, where $f(x)=-\log|T'(x)|$ is the geometric potential. However, since the only potential used in this paper is the geometric potential it is convenient to simply call this function the pressure of $T$ at $t\in\R$. 

 We denote by $\mathcal{M}_T$ the space of $T-$invariant probability measures on $I$. The entropy of the measure $\mu\in \mathcal{M}_T$ is denoted by $h(\mu)$ (see \cite[Chapter 4]{wa} for a precise definition). The pressure satisfies the following variational principle and approximation property (see \cite[Sub-section 2.1]{ij}).

\begin{proposition} \label{prop:pre} For every $t \in \R$ we have
\begin{eqnarray*}
P(t) &=& \sup \left\{ h(\nu) -t  \int \log |T'| \ d \nu : \nu \in \mathcal{M}_{T} \text{ and } \int \log |T'| \ d \nu < \infty \right\}\\
&=& \sup \{ P_{K}(t) : K\in \cK \},
\end{eqnarray*}
 where $\cK:= \{ K \subset [0,1] : K \ne \emptyset \text{ compact and }
T\text{-invariant}\}.$
\end{proposition}

There is a precise description of the regularity properties of the pressure (see \cite[Sub-sections 2.1 and 2.2]{ij}).

\begin{proposition} \label{bip}
There exists $s_{\infty} \in (0, \infty]$ such that pressure function $t \to P(t)$ has the following properties
\begin{equation*}
P(t)=
\begin{cases}
\infty  & \text{ if  } t  < s_{\infty} \\
\text{real analytic, strictly decreasing and strictly convex } & \text{ if  } t >s_{\infty}.
\end{cases}
\end{equation*}
Moreover, if $t> s_{\infty}$ then there exists a unique measure $\mu_t \in \mathcal{M}_T$, that we call \emph{equilibrium measure for $t$},  such that $P(t)=h(\mu_t) -t  \int \log |T'| \ d \mu_t$. 
\end{proposition}

It was noted by Bowen \cite{bo} in the finitely many branches setting (the compact case) that the pressure $P(t)$ captures a great deal of geometric information about $\Lambda$. This observation was first generalized  to the EMR setting by Mauldin and   Urba\'nski  in \cite[Theorems 3.15, 3.21, 3.24]{mu} (see also \cite[Theorem 3.1 and Proposition 3.1]{i}).

\begin{proposition} \label{prop:bow}
If $T$ is an EMR map then
\begin{equation*}
\dim_H(\Lambda)= \inf \left\{t \in \R	: P(t) \leq 0	\right\}.		
\end{equation*}
Moreover, if $s_{\infty} < \dim_H(\Lambda)$ there exists a unique ergodic measure $\nu \in \mathcal{M}_T$ such that
$\dim_H \nu = \dim_H \Lambda$. This measure is called \emph{measure of maximal dimension}.
 \end{proposition}

%

\subsection{Pressure and box dimension}\label{www}
In this Sub-section we prove that the number $s_{\infty}$ has a  dimension interpretation. It is a lower bound for the upper box dimension of the boundary points of the Markov partition. Moreover, if the box dimension of such set exists then it coincides with $s_{\infty}$. 

\begin{theorem} \label{main}
Let $T$ be an EMR map then
\begin{equation*}
s_{\infty}\le\overline{\dim}_B \left(\bigcup_{n=1}^{\infty} \left\{a_n, b_n	\right\} \right).
\end{equation*}
If the box dimension of $\bigcup_{n=1}^{\infty} \left\{a_n, b_n	\right\}$ exists, then   $$s_\infty=\dim_B\bigcup_{n=1}^{\infty} \left\{a_n, b_n	\right\}.$$
\end{theorem} 

\begin{proof}
Assume first that the map $T$ is piecewise linear. In this case the slope of $T$ restricted to the sub-interval $I_n=[a_n, b_n]$ satisfies $|T'|=(b_n -a_n)^{-1}$. In this situation the pressure function can be computed explicitly (see for example \cite[equation 9]{bi}). We have that
\begin{equation}\label{eq:pres}
P(t)= \log \sum_{n=1}^{\infty} (b_n -a_n)^{t}.
\end{equation}
Define $$\underline{L}=\liminf_{n \to \infty} \frac{\log n}{- \log (b_n - a_n)} \quad \text{ and }  \quad \overline{L}=\limsup_{n \to \infty} \frac{\log n}{- \log (b_n - a_n)}.$$
Given $\epsilon>0$, there exists $N\in \N$ such that if $n\ge N$ we have that 
$$(b_n-a_n)^{\overline{L}+\epsilon}<\frac{1}{n}<(b_n-a_n)^{\underline{L}-\epsilon}.$$
In particular, for $r>0$, we get that 
\begin{align}\label{eq:2} \sum_{n=N}^{\infty}    (b_n - a_n)^{r(\overline{L} + \epsilon)} \leq   \sum_{n=N}^{\infty} \frac{1}{n^r}\le  \sum_{n=N}^{\infty}    (b_n - a_n)^{r(\underline{L} - \epsilon)} \end{align}
Combining equation (\ref{eq:pres}) and inequality (\ref{eq:2}) we obtain that 
\begin{equation} \label{eq:3}
P(r(\overline{L} + \epsilon)) = \log \sum_{n=1}^{\infty}    (b_n - a_n)^{r(\overline{L} + \epsilon)} \leq \log\bigg(\sum_{n=1}^{N-1}(b_n-a_n)^{r(\overline{L}+\epsilon)}+ \sum_{n=N}^{\infty} \frac{1}{n^r}\bigg).
 \end{equation}
If $r>1$, then the right hand side of (\ref{eq:3}) converges. It follows from the definition of $s_\infty$ (see Proposition \ref{bip}) that $r(\overline{L}+\epsilon)\ge s_\infty$. Since $r$ is an arbitrary number larger than $1$, it follows that $\overline{L}+\epsilon\ge s_\infty$. Since $\epsilon$ is an arbitrary positive number, we conclude that $\overline{L}\ge s_\infty$. A similar argument using equation (\ref{eq:pres}) and the right hand side of inequality (\ref{eq:2}) give us that $\underline{L}\le s_\infty$. We obtained that 
\begin{equation}\label{eq:ineq} \underline{L}\le s_\infty\le \overline{L}. \end{equation}
Inequality (\ref{eq:ineq}) also holds in the general case; we can  reduce it to the linear case. Indeed, by the Mean Value Theorem for every $n \in \N$ there exists $x \in [a_n, b_n] $
such that $|T'(x)|= (b_n-a_n)^{-1}$. By the Jacobian estimate (see Lemma \ref{lem:je}) we have that if $y \in [a_n, b_n]$ then 
\begin{equation*}
\frac{(b_n-a_n)^{-1}}{C} \leq |T'(y)| \leq C (b_n - a_n)^{-1}. 
\end{equation*}
Therefore
\begin{equation*}
-t \log C + \log \sum_{n=1}^{\infty}    (b_n - a_n)^{t} \leq P(t) \leq t \log C + \log \sum_{n=1}^{\infty}    (b_n - a_n)^{t}.
\end{equation*}
Set $S:=\bigcup_{n=1}^{\infty} \left\{a_n, b_n	\right\}.$ Observe that by \cite[Proposition 3.6]{fa2} and  \cite[Proposition 3.7]{fa2} we know that \begin{equation*} \underline{L}\le \underline{\dim}_B(S)\le \overline{\dim}_B(S)\le \overline{L},\end{equation*} 
and that  
\begin{equation*} \frac{\underline{\dim}_B(S)(1-\overline{\dim}_B(S))}{(1-\underline{\dim}_B(S))}\le \underline{L}\le \overline{L}\le \overline{\dim}_B(S) .\end{equation*} 
In particular we have that $\overline{\dim}_B(S)=\overline{L}$. As mentioned in \cite[Corollary 3.8]{fa2}, the box dimension of $S$ exists if and only if $\underline{L}=\overline{L}$. It follows from this and inequality (\ref{eq:ineq}) that if $\underline{L}=\overline{L}$, then  $$s_\infty=\dim_B(S).$$
In general we only have the inequality $$s_\infty\le \overline{L}=\overline{\dim}_B(S).$$

\end{proof}

\begin{example}
If $G$ is the Gauss map (see Example \ref{ex:g}) then the set of boundary points is $\{1/n : n \in \N \}$. The box dimension of this set is equal to $1/2$ (see \cite[Example 3.5]{fa}) and  $s_{\infty}= 1/2$ (see \cite[p.150]{pw}).
\end{example}

\begin{remark} \label{iterates}
Note that if $T$ is and EMR map then, for any $k \geq 1$,  the map $T^k$ satisfies all assumptions of an EMR map except for condition $(a)$. If we denote by $P_{T^k}( \cdot)$ the pressure associated to the dynamical system $T^k$ then a classical result relates it to the pressure of $T$ (see \cite[Theorem 9.8 (i)]{wa}). Indeed, if $f:\Lambda \to \R$ is a  locally H\"older potential (see \cite[p.616]{irv} for precise definition) then
\begin{equation*}
P_{T^k}(S_k f) = kP_T(f),
\end{equation*}
where $S_k f$ is the Birkhoff sum of $f$. In particular, if $f= \log |T'|$ then 
\begin{equation*}
P_{T^k}(-t \log |(T^k)'|)= k P_T(-t \log|T'|).
\end{equation*} 
Therefore the number $s_{\infty}(T)$ corresponding to $T$ coincides with  $s_{\infty}(T^k)$ corresponding to $T^k$. Since condition $(a)$ in the definition of  EMR map is not used in the proof of Theorem \ref{main} we have $s_{\infty}$ is a lower bound for  the upper 
box dimension of the boundary points of the Markov partition of $T^k$, for any  $k \geq 1$. 
 \end{remark}

\subsection{Compact perturbations and measures of maximal dimension} 
The following is a consequence of Theorem \ref{main} and Proposition \ref{prop:bow}. 

\begin{corollary}
Let $T$ be an EMR map. If $\overline{\dim}_B \left(\bigcup_{n=1}^{\infty} \left\{a_n, b_n	\right\} \right) < \dim_H(\Lambda)$ then there exists a measure of maximal dimension.
\end{corollary}
We now define compact perturbations of the map $T$.

\begin{definition}
Let $T$ be an EMR map defined on the sequence of closed intervals $(I_n)_n$.  We say that $\tilde{T}$ is a \emph{compact perturbation} of $T$ if the following two conditions are satisfied:
\begin{enumerate}
\item  There exists a compact subset 
 $K \subset (0, 1]$ with the property that if $\text{int }I_n \cap K \neq \emptyset$ then $I_n \subset K$ and $T(x)=\tilde{T}(x)$ for every $x \in \left( \bigcup_{i=1}^{\infty} I_i \right) \setminus K$.
\item The map $\tilde{T}$ is an EMR map.
\end{enumerate}
\end{definition}

The following result shows that the behaviour of the pressure at  $s_{\infty}$ not only determines the existence of measures of maximal dimension for the map $T$, but also for compact perturbations of it.

\begin{corollary}\label{cor:pert}
Let $T$ be an EMR map  such that $P(s_{\infty})=\infty$. Then any compact perturbation $\tilde{T}$ of $T$ has a measure of maximal dimension.
\end{corollary}

\begin{proof}
Note that for any compact perturbation $\tilde{T}$ the number $s_{\infty}$ corresponding to $\tilde{T}$ is equal to that of $T$. Moreover, the pressure evaluated at that point is infinity in both cases. Therefore $\tilde{T}$ has a measure of maximal dimension.
\end{proof}

\begin{remark}
If $T$ is an EMR map with  $s_{\infty} < \dim_H(\Lambda)$ and $P(s_{\infty}) < \infty$ then there exists compact perturbations without measures of maximal dimension. An example can be constructed considering the partition of $[0,1]$ given by the sequence defined by $a_n=  1/ (n (\log n)^2)$, see \cite[Example 3.1]{i}.
\end{remark}

\subsection{Extremes of the multifractal spectrum}

The following result is a consequence of Theorem \ref{main} and results obtained by Fan, Jordan, Liao and Rams \cite{fjlr}.

\begin{corollary}
Let $T$ be an EMR map for which the sequence $\big(\frac{\log n}{-\log(b_n-a_n)}\big)_n$ converges. Let $\phi: \Lambda \to \R$ a bounded H\"older potential. Then for every 
\begin{equation*}
\alpha \in \left[ \inf \left\{ \int \phi d \mu : \mu \in \mathcal{M}_T \right \}  , \sup \left\{ \int \phi d \mu : \mu \in \mathcal{M}_T \right\}	\right]
\end{equation*}
we have that
\begin{equation*}
\dim_H \left( \left\{ x \in \Lambda : \lim_{n \to \infty} \frac{1}{n} \sum_{i=0}^{n-1} \phi(T^i x) = \alpha \right\} \right) \geq  \dim_B \left(\bigcup_{n=1}^{\infty} \left\{a_n, b_n	\right\} \right).
\end{equation*}
\end{corollary}

The next result is a consequence of Theorem \ref{main} and \cite[Theorem 7.1]{ij}.

\begin{corollary}\label{cor:lim}
Let $T$ be an EMR map for which the sequence $\big(\frac{\log n}{-\log(b_n-a_n)}\big)_n$ converges. Then 
\begin{equation*}
\dim_H \left( \left\{ x \in \Lambda : \lim_{n \to \infty} \frac{1}{n} \sum_{i=0}^{n-1} \log|T^i x| = \infty \right\} \right) =  \dim_B \left(\bigcup_{n=1}^{\infty} \left\{a_n, b_n	\right\} \right).
\end{equation*}
\end{corollary}

%

\section{Geodesic flow on a negatively curved manifold}  \label{sec:geo}

In this section we will relate the box dimension of points at the Gromov boundary of Hadamard manifolds with the  critical exponent of parabolic subgroups. As explained in the introduction the critical exponent of a parabolic subgroup is related to the entropy at infinity of the geodesic flow on a geometrically finite manifold; our results relate the entropy at infinity with dimension theory of points at the Gromov boundary. A good reference for the facts used in this section is \cite{bh}. 

Let $(X,g)$ be a pinched negatively curved Hadamard manifold. We will moreover assume that the sectional curvature of $X$ is bounded above by $-1$, in other words, that $X$ is a $Cat(-1)$ space. The main example of a $Cat(-1)$ space is hyperbolic space $\mathbb{H}^n$, where the sectional curvature is constant equal to $-1$. 

The Gromov boundary of $X$ is denoted by $\partial X$. The Gromov product of $\xi\in\partial X$ and $\xi'\in\partial X$ with respect to the point $x\in X$ is defined as 
$$(\xi|\xi')_x=\frac{B_\xi(x,z)+B_{\xi'}(x,z)}{2},$$
where $B_\xi(p,q)$ is the Busemann function and $z$ is a point in the geodesic connecting $\xi$ and $\xi'$. The Gromov product is independent of the choice of the point $z$.  From now on we fix a reference point $o\in X$ and we use the notation $(\xi|\xi'):=(\xi|\xi')_o$. Define $d_{\partial X}:\partial X\times \partial X\to \R_{\ge 0}$, by the formula
\begin{equation*}
  d_{\partial X} (\xi,\xi')=
 \begin{cases} 
 e^{-(\xi|\xi')}  & \text{ if }\xi\ne\xi', \\
      0  & \text{ if }\xi=\xi'.
\end{cases}
\end{equation*}
\begin{proposition}\cite{bou} The function $d_{\partial X}:\partial X\times \partial X\to \R_{\ge 0}$ is a metric on $\partial X$. \end{proposition}

The metric $d_{\partial X} $ was introduced by Bourdon in \cite{bou} and defines a canonical conformal structure on $\partial X$ (which is independent of the reference point $o$).

The isometry group of $(X,g)$ is denoted by $\text{Iso}(X)$. It is well known that the action of every element $g\in \text{Iso}(X)$  on $X$, extends to a homeomorphism of the Gromov boundary $\partial X$. An element $g\in \text{Iso}(X)$ is called \emph{hyperbolic} if its action on $\partial X$ has two fixed points.  An element $g\in \text{Iso}(X)$ is called \emph{parabolic} if its action on $\partial X$ has a unique fixed point. We say that a group $G\leqslant\text{Iso}(X)$ is  parabolic if there exists a unique point $w\in \partial X$ that is fixed by the action of every element in $G$. We say that $G\leqslant\text{Iso}(G)$ is \emph{elementary} if it is a parabolic subgroup or generated by a single hyperbolic element. We say that $G$ is \emph{non-elementary} if it is not elementary.

To a group of isometries $G\leqslant\text{Iso}(X)$ we can associate a non-negative number, the so-called critical exponent of $G$. A fundamental property of the critical exponent is the following: if $G$ acts free and properly discontinuous on $X$ and $G$ is non-elementary, then the  topological entropy of the geodesic flow on $X/G$ is equal to the critical exponent of $G$ (see \cite[Theorem 1]{op}). 

\begin{definition}
Let $G\leqslant\text{Iso}(X)$ be a group of isometries. The \emph{Poincar\'e series} of G is defined by
\begin{equation*}
\mathcal{P}(s)=\sum_{g\in G}e^{-sd(o,g o)}.
\end{equation*}
The \emph{critical exponent} $\delta_G$ is defined by
\begin{equation*}
\delta_{G}:= \inf \left\{	s \in \R: \mathcal{P}(s) < \infty	\right\}.
\end{equation*}
\end{definition}

We remark that the critical exponent of $G$ is independent of the reference point $o\in X$. Given a geodesic triangle with vertices $x$, $y$ and $z$ we associate a comparison hyperbolic triangle with vertices $A$, $B$ and $C$ in $\H^2$, such that $d_{\mathbb{H}^2}(A,B)=d(x,y)$, $d_{\mathbb{H}^2}(B,C)=d(y,z)$ and $d_{\mathbb{H}^2}(C,A)=d(z,x)$, where $d$ is the Riemannian distance on $X$ and $d_{\mathbb{H}^2}$ the Riemannian distance on $\mathbb{H}^2$. The angle at $x$ of the geodesic triangle $xyz$ is defined as the angle at $A$ of the hyperbolic triangle $ABC$ and it is denoted by $\angle_x (y,z)$. The following lemma follows directly from the hyperbolic law of cosines.

\begin{lemma}\label{geoine} Given $D>0$, there exists a constant $C=C(D)>0$, such that for every geodesic triangle with vertices $x, y, z$, and angle at $z$ bigger than $D$, then $$d(x,y)\ge d(z,x)+d(z,y)-C.$$
\end{lemma}

We will now verify that the same box dimension interpretation obtained for EMR maps in Theorem \ref{main} holds for the action of parabolic groups on the Gromov boundary $\partial X$. We will first deal with the case when $X$ is a surface. 
 
\begin{proposition}\label{parabolic} Let $(X,g)$ be a $Cat(-1)$ surface. Let  $p\in \text{\normalfont Iso}(X)$ be a parabolic isometry and $\xi\in \partial X$ a point that is not fixed by $p$.  Then $$\liminf_{k\to\infty}\frac{\log k}{  -\log d_{\partial X}(p^k\xi,p^{k+1}\xi)}\le \delta_{\langle p\rangle}\le \limsup_{k\to\infty}\frac{\log k}{  -\log d_{\partial X}(p^k\xi,p^{k+1}\xi)}.$$
If the sequence $\left(\frac{\log k}{  -\log d_{\partial X}(p^k\xi,p^{k+1}\xi)}\right)_k$ converges as $k$ goes to infinity, then 
$$\lim_{k\to\infty}\frac{\log k}{  -\log d_{\partial X}(p^k\xi,p^{k+1}\xi)}=\delta_{\langle p\rangle}.$$
\end{proposition}

\begin{proof}
We denote by $\eta$  the fixed point of $p$.  Choose a point $z$ in the geodesic connecting $\xi $ and $p\xi $. Observe that $p^kz$ belongs to the geodesic connecting $p^k\xi $ and $p^{k+1}\xi $. Note that 
\begin{equation*}
(p^k\xi |p^{k+1}\xi )=\frac{1}{2} \left(B_{p^k\xi }(o,p^kz)+B_{p^{k+1}\xi }(o,p^kz) \right)=\frac{1}{2} \left(B_{\xi }(p^{-k}o,z)+B_{p\xi }(p^{-k}o,z) \right).
\end{equation*}
Since $B_q(x,y)\le d(x,y)$, for all $q\in\partial X$ and $x,y\in X$ we conclude that 
\begin{align}\label{ine:1} (p^k\xi |p^{k+1}\xi )\le d(p^{-k}o,z)\le d(p^{-k}o,o)+d(o,z)= d(o,p^ko)+d(o,z).\end{align}
Let $\alpha(t)$ be a parametrization of the geodesic ray starting at $z$ and converging to $\xi $. If $|k|$ is sufficienly large, then $p^{-k}o$ will be  in a small neighborhood of  $\eta$ in $X\cup \partial X$.  In particular if $|k|$ and $t$ are sufficiently large, then the angle at $z$ of the geodesic triangle with vertices $z$, $p^{-k}o$, and $\alpha(t)$ is uniformly bounded below (it will be close to  the angle at $z$ between the geodesic rays $[z,\xi )$ and $[z,\eta)$). It follows from Lemma \ref{geoine}  that for sufficiently large values of  $t$ and $|k|$ we have
$$d(\alpha(t),p^{-k}o)-d(\alpha(t),z)\ge d(p^{-k}o,z)-C,$$
for some positive constant $C=C(\xi, \eta)$. By definition of the Busemann function we obtain that 
\begin{align*}
B_\xi(p^{-k}o,z)=\lim_{t\to\infty}  d(\alpha(t),p^{-k}o)-d(\alpha(t),z)&\ge d(p^{-k}o,z)-C\\
&\ge d(p^{-k}o,o)-d(o,z)-C
\\&=d(o,p^ko)-d(o,z)-C. \end{align*}
An analogous argument gives that there exists a constant $C'>0$ (independent of $k$) such that for $|k|$ sufficiently large we have 
$$B_{p\xi}(p^{-k}o,z)\ge d(p^{-k}o,z)-C'\ge d(o,p^ko)-d(o,z)-C'.$$
We conclude that for $|k|$ sufficiently large we have 
\begin{align}\label{ine:2}(p^k\xi|p^{k+1}\xi)\ge d(o,p^ko)-c,\end{align}
for some constant $c>0$ independent of $k$. In order to conclude the proof we will need the following fact. 

\begin{lemma} The following inequality holds 
\begin{align}\label{ine:crit}\liminf_{k\to \infty} \frac{\log k}{d(o,p^ko)} \le \delta_{\langle p\rangle}\le\limsup_{k\to \infty} \frac{\log k}{d(o,p^ko)}.\end{align}
If the sequence $\left(\frac{\log k}{d(o,p^k o)} \right)_k$ converges as $k$ goes to infinity, then 
$$\delta_{\langle p\rangle}=\lim_{k\to \infty} \frac{\log k}{d(o,p^ko)}.$$
\end{lemma}

\begin{proof} Observe that if $\limsup_{n\to\infty}\frac{\log n}{\log a_n^{-1}}<1$, then the series $\sum_n a_n$ converges. Applying this fact to $a_n=e^{-sd(o,p^n o)}$ we obtain that if $$\limsup_{n\to\infty}\frac{\log n}{d(o,p^n o)}<s,$$ then $\mathcal{P}(s)$ converges. This immediately implies that $\delta_{\langle p\rangle}\le \limsup_{n\to\infty}\frac{\log n}{d(o,p^n o)}$. Similarly, if 
$\liminf_{n\to\infty}\frac{\log n}{\log a_n^{-1}}>1$, then the series $\sum a_n$ diverges. Applying this to $a_n=e^{-sd(o,p^n o)}$ we obtain that $$\delta_{\langle p\rangle}\ge \liminf_{n\to\infty}\frac{\log n}{d(o,p^n o)}.$$
\end{proof}

Inequalities (\ref{ine:1}) and (\ref{ine:2}) imply that in  inequality (\ref{ine:crit}) we can replace $d(o,p^k o)$ by $-\log d_{\partial X}(p^k\xi, p^{k+1}\xi)=(p^k\xi|p^{k+1}\xi)$.
\end{proof}

 If $(X,g)$ is the hyperbolic disk, then $d_{\partial X}$ is related to the angle at $o$ between the geodesic rays $[o,\xi)$ and $[o,\xi')$. For now on we assume that $o$ is the origin of the hyperbolic disk. Bourdon \cite{bou}  proved  that $$d_{\partial \mathbb{H}}(\xi,\xi')=\sin\frac{1}{2}\angle_o(\xi,\xi'),$$
where $\angle_o(\xi,\xi')$ is the angle (using radians) between the (straight) rays $[o,\xi)$ and $[o,\xi')$. The euclidean metric on $\R^2$ induces a metric on $\partial \mathbb{H}$. We denote such metric by $d_1$, and it is given by the formula $$d_1(\xi,\xi')=\angle_o (\xi,\xi').$$ 

\begin{corollary}\label{cor:dim2} Let $\xi\in \partial \mathbb{H}$ and $p\in \text{\normalfont Iso}(\H^2)$ a parabolic isometry.  Assume that $\xi$ is not the fixed point of $p$.  Then
$$\delta_{\langle p\rangle}= \dim_B \left( \bigcup_{k\in \Z}p^k\xi  \right)=\frac{1}{2},$$
where the box dimension is computed using the spherical metric on $\partial \H^2$.
\end{corollary}
\begin{proof}
Since $\lim_{x\to0}\frac{\sin x}{x}=1$, and $\lim_{k\to \infty}\angle_o(p^k\xi,p^{k+1}\xi)=0$, one can easily check that in  Proposition \ref{parabolic} it is possible to  replace $d_{\partial X}$ by $d_1$. In the hyperbolic disk it is known that the sequence $(d(o,p^no)-2\log n)_n$ is bounded (see proof of Lemma \ref{lem:k/2}). In particular we know that $\lim_{k\to\infty}\frac{\log k}{d(o,p^ko)}=\frac{1}{2}$. This implies that equality holds in inequality (\ref{ine:crit}).  Since the metric $d_1$ is equivalent to the flat metric on $\R$ we conclude the desired result.
\end{proof}

As mentioned in the introduction, one of the main results in this paper is the generalization of Corollary \ref{cor:dim2} to higher dimensions. From now on we will always assume that $\Gamma\leqslant \text{Iso}(\H^n)$ is a discrete, torsion free subgroup of isometries. Fix a reference point $o\in \H^n$. We define the \emph{limit set of} $\Gamma$ as the set of accumulation points of $\Gamma.o$ in $\H^n\cup \partial \H^n$, and  denoted it by $\Lambda(\Gamma)$. The limit set of $\Gamma$ is a subset of $\partial \H^n$ and it is independent of the base point $o$. If $\Gamma$ is non-elementary we have a very nice characterization of $\Lambda(\Gamma)$: it is the minimal $\Gamma$-invariant closed subset of $\partial \H^n$  (for instance see \cite{bh}). 

The non-wandering set of the geodesic flow on $\H^n/\Gamma$ is denoted by $\Omega\subset T^1(\H^n/\Gamma)$. We say that $\Gamma\leqslant\text{Iso}(\H^n)$ is \emph{geometrically finite} if an $\epsilon$-neighborhood of $\Omega$ has finite Liouville measure. Several equivalent definitions of geometrically finiteness are given in \cite{bowd2}.  We will need the following result of Stratmann and Urba\'nski (see \cite[Theorem 3]{su}). 
\begin{theorem}\label{su} Assume that $\Gamma\leqslant\text{\normalfont Iso}(\H^n)$ is geometrically finite. Then $$\delta_\Gamma=\dim_H\Lambda(\Gamma)=\dim_B \Lambda(\Gamma),$$
where the box and Hausdorff dimension are computed using the spherical metric on $\partial \H^n$. 
\end{theorem}
We remark that for hyperbolic geometrically finite manifolds the equality $\delta_\Gamma=\dim_H\Lambda(\Gamma)$, was proved by Sullivan (see \cite[Theorem 25]{s}). This result was latter generalized by Bishop and Jones to cover all hyperbolic manifolds (see \cite[Theorem 2.1]{bj}). More precisely, they proved that  $\delta_\Gamma=\dim_H\Lambda_{rad}(\Gamma)$, where $\Lambda_{rad}(\Gamma)$ is the radial limit set of $\Gamma$ (for precise definitions we refer the reader to \cite{bj}). If $G$ is a group we define $G^*$ to be $G\setminus\{id\}$. We will need the following definition. 
\begin{definition} Let $F_1$ and $F_2$ be discrete, torsion free subgroups of $\text{Iso}(\H^n)$. We say that $F_1$ and $F_2$ are in Schottky position if there exist disjoint closed subsets $U_{F_1}$ and $U_{F_2}$ of $\partial\H^n$ such that $F_1^*(\partial \H^n\setminus U_{F_1})\subset U_{F_1}$ and $F_2^*(\partial \H^n\setminus U_{F_2})\subset U_{F_2}$. 
\end{definition}

 We now state one of the main results of this section.

\begin{theorem}\label{hdim} Let $\cP\leqslant\text{\normalfont Iso}(\H^n)$ be a parabolic subgroup and $\xi\in \partial\H^n$ a point not fixed by $\cP$. Then $$\delta_\cP=\dim_B \bigg(\bigcup_{p\in \cP} p\xi\bigg),$$
where the box dimension is computed using the spherical metric on $\partial \H^n$. 
\end{theorem}

\begin{remark}
 Theorem \ref{hdim} does not immediately follow from Theorem \ref{su} since the group $\cP$ is not geometrically finite. Moreover, the group $\cP$ is elementary and $\Lambda(\cP)=\{\eta\}$, where $\eta$ is the point fixed by $\cP$. Nevertheless, using box dimension we can recover the critical exponent from the orbit of $\xi$ under $\cP$. 
\end{remark}

\begin{proof} We first prove the inequality $\overline{\dim}_B \cP\xi\le \delta_\cP$. There exists a  hyperbolic isometry $h$ which fixes $\xi$ and such that  $\langle h\rangle$ and $\cP$ are in Schottky position. Indeed, let $\cV_\cP$ be a fundamental domain of the action of $\cP$ on $\partial \H^n$ such that $\xi\in \text{int }\cV_\cP$, and define $\cU_\cP=\partial\H^n\setminus \cV_\cP$. Observe that for every $p\in \cP\setminus\{id\}$ we have that $p(\partial \H^n\setminus \cU_\cP)\subset \cU_\cP$. Now choose a hyperbolic isometry $w$  such that $w(\xi)=\xi$, and such that the other fixed point of $w$, say $\xi_0$, also belongs to $\text{int }\cV_\cP$. If one chooses $k$ large enough, then $h=w^k$ will satisfy the properties stated above: if $k$ is large enough, then we can find a neighborhood $\cU_h$ of $\{\xi,\xi_0\}$ such that $\cU_h\subset \cV_\cP=\partial \H^n\setminus \cU_\cP$ and $h^s(\partial\H^n \setminus\cU_h)\subset \cU_h$, for every $s\ne 0$. This implies that $\langle h\rangle$ and $\cP$ are in Schottky position.

Define $\Gamma_k=\cP \ast \langle h^k\rangle$. We can now use \cite[Proposition 5.3]{irv} to conclude that $\lim_{k\to\infty}\delta_{\Gamma_k}=\delta_\cP$. Moreover, the manifold $\H^n/\Gamma_k$ is geometrically finite (see \cite{dp} for a proof when $\cP$ has rank one, but the same argument applies to $\Gamma_k$). By construction $\xi\in \Lambda(\Gamma_k)$, therefore $\Gamma_k \xi\subset \Lambda(\Gamma_k)$; in particular $\overline{\Gamma_k \xi}\subset \Lambda(\Gamma_k)$. Observe that $\overline{\Gamma_k \xi}$ is a closed $\Gamma_k$-invariant subset of $\partial \H^n$. Since  $\Gamma_k$ is non-elementary we conclude that $\overline{\Gamma_k \xi}=\Lambda(\Gamma_k)$. We obtained that  $\cP\xi\subset \Gamma_k\xi\subset \overline{\Gamma_k \xi}=\Lambda(\Gamma_k)$. Using Theorem \ref{su} and this inclusion we  conclude that 
$$\overline{\dim}_B \cP\xi\le \overline{\dim}_B\Lambda(\Gamma_k)=\dim_B\Lambda(\Gamma_k)=\delta_{\Gamma_k},$$
for every $k\in \N$. Therefore $$\overline{\dim}_B\cP\xi\le \lim_{k\to\infty}\delta_{\Gamma_k}=\delta_\cP.$$
 
We  now prove the inequality $\underline{\dim}_B\cP \xi \ge \delta_\cP$. It will be convenient to consider the upper half-space model of hyperbolic space. We identify $\H^n$ with $\{(x_1,...,x_n)\in\R^{n}:x_n>0\}$, and the fixed point of $\cP$ with the point of $\partial \H^n$ whose $x_n$--coordinate is $\infty$ (the rest of $\partial \H^n$ is the plane $x_n=0$). Our reference point will be $o=(0,...,0,1)$ (this point corresponds to the origin in the ball model). Via the action of $\cP$ on the horosphere $x_n=1$, we can identify $\cP$ with a subgroup of $\text{Aff}(\R^{n-1})$, where $\text{Aff}(\R^{n-1})$ is the space of affine transformations of $\R^{n-1}$. It follows from Bieberbach's theorem that there exists a finite index subgroup $\cP_0$ of $\cP$ such that $\cP_0$ acts  on $x_n=1$ as a group of translations (for instance see \cite[Theorem 1.1]{bowd}). In particular $\cP_0\cong \Z^k$, for some $k\in \{1,...,n-1\}$. The number $k$ is called the rank of $\cP$. It is well known that $\delta_\cP=\delta_{\cP_0}=\frac{k}{2}$ (see for instance \cite[Section 3]{dop}). There exists a sequence $\{g_1,...,g_m\}\subset \cP$ such that $\cP=\bigcup_{i=1}^m \cP_0g_i$, in particular $\cP\xi=\bigcup_{i=1}^m \cP_0g_i\xi$. We will prove that 
\begin{align}\label{ineB} \underline{\dim}_B\cP_0\eta \ge \frac{k}{2},\end{align}
for every $\eta\in\partial \H^n$.  Assume that inequality (\ref{ineB}) holds. As mentioned before $\delta_{\cP_0}=\frac{k}{2}$. Combining inequality (\ref{ineB}) with the inequality $\overline{\dim}_B\cP_0\eta\le \delta_{\cP_0}$, we obtain that $\dim_B\cP_0\eta=\frac{k}{2}=\delta_{\cP_0}=\delta_\cP$, for every $\eta\in \partial\H^n$. Since the box dimension is finitely additive we can conclude that $\dim_B\cP\xi=\frac{k}{2}=\delta_\cP$ (recall that   $\cP\xi=\bigcup_{i=1}^m \cP_0g_i\xi$). 

We now prove inequality (\ref{ineB}). We will identify the group $\cP_0$ with $\Z^k$, and use the notation $n.\xi$ to denote the translate of $\xi$ under $n\in\Z^k$. The standard generators of $\Z^k$ will be denoted by $\{e_1,...,e_k\}$.  Denote by $\rho$ the fixed point of $\cP_0$. 
Since $\lim_{k\to\infty}k.\eta=\rho$, we can conclude that there exists $C>0$ such that the geodesic connecting $\eta$ and $k.\eta$ is at distance at most $C$ from $o$. We choose $z_k\in \H^n$ in the geodesic connecting $\eta$ and $k.\eta$ such that $d(o,z_k)\le C$.  Observe that $N.z_k$ belongs to the geodesic connecting $N.\eta$ and $(N+k).\eta$, therefore
\begin{align*}((N+k).\eta|N.\eta)=&\frac{1}{2}(B_{(N+k).\eta}(o, N.z_k)+B_{N.\eta}(o, N.z_k))\\
=&\frac{1}{2}(B_{k.\eta}((-N).o, z_k)+B_{\eta}((-N).o, z_k))\\
\le&d((-N).o,z_k)\\
\le& d(o,N.o)+d(o.z_k)\\
\le& d(o,N.o)+C.
\end{align*}
We conclude that $$d_{\partial \H^n}((N+k).\eta,N.\eta)\ge e^{-d(o,N.o)-C}.$$
In order to conclude the proof we will need the following result.  

\begin{lemma}\label{lem:k/2} Using the notation above we have $$\lim_{t\to\infty}\frac{\#\log\{N\in \Z^k: d(o,N.o)\le t\}}{t}=\frac{k}{2}.$$
\end{lemma}

\begin{proof} Recall that $o=(0,...,0,1)$ is the origin of hyperbolic space.  Let $\alpha_i\in \R^{n-1}$ be the translation vector of $e_i$ while acting on the horosphere $x_n=1$. If follows from the definition of the hyperbolic metric on $\H^n$ that $$d(o,N.o)=2\text{arcsinh }\bigg(\frac{1}{2}|N_1\alpha_1+...+N_k\alpha_k|\bigg),$$
where $|.|$ stands for the euclidean metric on $\R^{n-1}$. It follows that the set $$\{d(o,N.o)-\log(|N_1\alpha_1+...+N_k\alpha_k|^2):N=(N_1,...,N_{k})\in \Z^{k}\},$$ is bounded. It is enough to prove that \begin{align}\label{lem:lim}\lim_{t\to\infty}\frac{\#\log\{N\in \Z^k: \log(|N_1\alpha_1+...+N_k\alpha_k|^2) \le t\}}{t}=\frac{k}{2}.\end{align}
There exists $D=D(\alpha_1,...,\alpha_n)>0$ such that the set $A_t:=\{N\in \R^n:|N_1\alpha_1+...+N_k\alpha_k|\le e^{t/2}\},$ contains a cube centered at the origin of radius $D^{-1} e^{\frac{t}{2}}$, and it is contained in a cube centered at the origin of radius $D e^{\frac{t}{2}}$. For $t$ large the number of integral points in $A_t$ will be of the order $e^{tk/2}$. This immediately implies the equality (\ref{lem:lim}).
\end{proof}

It follows from  Lemma \ref{lem:k/2}   that given $\epsilon>0$, there exists $E=E(\epsilon)$ such that for $t\ge E$ we have 
$$\#\{N\in \Z^k: d(o,N.o)\le t\}\ge e^{(\frac{k}{2}-\epsilon)t}.$$
Assume that  $r$ is sufficiently small in order to have $\log(e^{-C}/r)\ge E$, and define $A_r:=\{N\in \Z^k: d(o,N.o)\le \log(e^{-C}/r)\}$. We conclude that  
$$|A_r|=\#\{N\in \Z^k: d(o,N.o)\le \log(e^{-C}/r)\}\ge \frac{e^{-C(\frac{k}{2}-\epsilon)}}{r^{\frac{k}{2}-\epsilon}}.$$
Observe that inequality $d(o, N.o)\le \log(e^{-C}/r)$, is equivalent to $r\le e^{-d(o,N.o)-C}$. Therefore $N\in A_r$ implies that for all $k\in \Z^k$ we have
\begin{align*} r\le  e^{-d(o,N.o)-C}&\le d_{\partial \H^n}((N+k).\eta,N.\eta)=\sin\frac{1}{2}d_1((N+k).\eta,N.\eta)\\ &\le \frac{1}{2}d_1((N+k).\eta,N.\eta)
\end{align*}
where $d_1$ is the spherical metric on $\partial \H^n$ (coming from the natural embedding of the ball model into $\R^n$).  In other words, to cover $\cP_0.\eta$ with $d_1$-balls of radius $r$ we need at least $|A_r|$ balls (at least one ball for each $N.\eta$, where $N\in A_r$). We conclude that $$\liminf_{r\to 0}\frac{\log N(r)}{\log(\frac{1}{r})}\ge \liminf_{t\to 0}\frac{\log |A_r|}{\log(\frac{1}{r})} \ge \liminf_{r\to 0} \frac{\log\bigg(\frac{e^{-C(\frac{k}{2}-\epsilon)}}{r^{\frac{k}{2}-\epsilon}}\bigg)}{\log\big(\frac{1}{2r}\big)}= \frac{k}{2}-\epsilon.$$
Since $\epsilon$ was arbitrary we conclude that $\underline{\dim}_B \cP_0\eta\ge  \frac{k}{2}.$ 

\end{proof}

\begin{remark} Roblin  proved in \cite{rob} that if $\Gamma$ is non-elementary, then  
$$\lim_{t\to\infty}\frac{\#\log\{\gamma\in \Gamma: d(o,\gamma o)\le t\}}{t}=\delta_\Gamma.$$ 
 Since  $\cP_0$ is elementary we can not use Roblin's result. The proof of Roblin uses in an essential way that the group is non-elementary--it uses the Patterson-Sullivan conformal density at infinity associated to $\Gamma$. 
\end{remark}


\begin{thebibliography}{XXX}


\bibitem[BI]{bi} Barreira, Luis; Iommi, Godofredo. \emph{Suspension flows over countable Markov shifts}. J. Stat. Phys. 124 (2006), no. 1, 207--230.

\bibitem[BJ]{bj} Bishop, Christopher; Jones, Peter. \emph{Hausdorff dimension and Kleinian groups}. Acta Math. 179 (1997), no. 1, 1--39. 

\bibitem[Bou]{bou} Bourdon, Marc. \emph{Structure conforme au bord et flot g\'eod\'esique d'un CAT(-1)-espace}. (French)  Enseign. Math. (2) {\bf 41} (1995), no. 1-2, 63--102. 

\bibitem[Bo1]{bowd} Bowditch, Brian. \emph{Discrete parabolic groups.} J. Differential Geom. 38 (1993), no. 3, 559--583. 

\bibitem[Bo2]{bowd2} Bowditch, Brian. \emph{Geometrical finiteness with variable negative curvature}. Duke Math. J. 77
(1995), 229--274.

\bibitem[B]{bo} Bowen, Rufus. \emph{Hausdorff dimension of quasicircles}. Inst. Hautes Etudes Sci. Publ. Math. No. 50 (1979), 11--25.

\bibitem[BH]{bh} Bridson, Martin; Haefliger, Andr\'e. \emph{Metric spaces of non-positive curvature.} Grundlehren der Mathematischen Wissenschaften, 319. Springer-Verlag, Berlin, 1999. xxii+643 pp.

\bibitem[CFS]{cfs}   Cornfeld, I.P. ; Fomin, S. V.; Sinai, Ya. G.  \emph{Ergodic theory.}  Grundlehren der Mathematischen Wissenschaften, 245. Springer-Verlag, New York, 1982. x+486 pp. 
 
 \bibitem[DP]{dp}   Dal'bo, Francoise; Peign\'e, Marc. \emph{Some negatively curved manifolds with cusps, mixing and counting.} J. Reine Angew. Math. 497 (1998), 141--169.

\bibitem[DOP]{dop} Dal'bo, Francoise; Otal, Jean-Pierre; Peign\'e, Marc. \emph{S\'eries de Poincar\'e des groupes g\'eom\'etriquement finis.} Israel J. Math. 118 (2000), 109--124.

\bibitem[Fa1]{fa} Falconer, Kenneth. \emph{Fractal geometry. Mathematical foundations and applications.} John Wiley and Sons, Ltd., Chichester, 1990. xxii+288 pp.

\bibitem[Fa2]{fa2} Falconer, Kenneth. \emph{Techniques in fractal geometry}. John Wiley \& Sons, Ltd., Chichester, 1997. xviii+256 pp.

\bibitem[FJLR]{fjlr}  Fan, Ai-Hua; Jordan, Thomas; Liao, Lingmin; Rams, Michal. \emph{Multifractal analysis for expanding interval maps with infinitely many branches.} Trans. Amer. Math. Soc. 367 (2015), no. 3, 1847--1870.



\bibitem[I]{i} Iommi, Godofredo \emph{Multifractal analysis for countable Markov shifts.} Ergodic Theory Dynam. Systems 25 (2005), no. 6, 1881--1907. 

\bibitem[IJ]{ij} Iommi, Godofredo; Jordan, Thomas. \emph{Multifractal analysis of Birkhoff averages for countable Markov maps.} Ergodic Theory Dynam. Systems 35 (2015), no. 8, 2559--2586.


\bibitem[IRV] {irv} Iommi, Godofredo; Riquelme, Felipe; Velozo, Anibal. \emph{Entropy in the cusp and phase transitions for geodesic flows}  Israel J. Math. 225 (2018), no. 2, 609--659.

\bibitem[ITV1]{itv1}  Iommi, Godofredo; Todd, Mike; Velozo, Anibal.       \emph{Upper semi-continuity of entropy in non-compact settings.} To appear in Math. Res. Lett. 

\bibitem[ITV2]{itv2}  Iommi, Godofredo; Todd, Mike; Velozo, Anibal.       \emph{Escape of mass and entropy for countable Markov shifts.} In preparation.




\bibitem[MU]{mu} Mauldin, R. Daniel; Urba\'nski, Mariusz. \emph{Dimensions and measures in infinite iterated function systems.}  Proc. London Math. Soc. (3)  73  (1996),  no. 1, 105--154.




\bibitem[OP]{op} Otal, Jean-Pierre; Peign\'e, Marc. \emph{Principe variationnel et groupes kleiniens}. Duke Math. J. 125 (2004), no. 1, 15--44.

\bibitem[PPS]{pps}   Paulin, Fr\'ed\'eric; Pollicott, Mark; Schapira, Barbara. \emph{Equilibrium states in negative curvature.} Asterisque No. 373 (2015), viii+281 pp.

\bibitem[PS]{ps}  Pit, Vincent;  Schapira, Barbara. \emph{Finiteness of gibbs measures on noncompact manifolds with pinched
negative curvature}. Ann. Inst. Fourier (Grenoble) 68 (2018), no. 2, 457--510.

\bibitem[PW]{pw} Pollicott, Mark; Weiss, Howard. \emph{Multifractal analysis of Lyapunov exponent for continued fraction and Manneville-Pomeau transformations and applications to Diophantine approximation.} Comm. Math. Phys. 207 (1999), no. 1, 145--171.

\bibitem[RV]{rv}  Riquelme, Felipe; Velozo, An\'ibal.   \emph{Escape of mass and entropy for geodesic flows}. To appear in Ergodic Theory Dynam. Systems. 

\bibitem[Rob]{rob} Roblin, Thomas. \emph{Sur la fonction orbitale des groupes discrets en courbure n\'egative}. Ann. Inst. Fourier (Grenoble) 52 (2002), no. 1, 145--151.

%
%

\bibitem[Sar]{sa}  Sarig, Omri. \emph{Existence of Gibbs measures for countable Markov shifts.} Proc. Amer. Math. Soc. 131 (2003), no. 6, 1751--1758.

\bibitem[SU]{su} Stratmann, Bernd Otto; Urba\'nski, Mariusz. \emph{The box-counting dimension for geometrically finite Kleinian groups.} Fund. Math. 149 (1996), no. 1, 83--93. 

\bibitem[Sul]{s} Sullivan, Dennis. \emph{The density at infinity of a discrete group of hyperbolic motions.} Inst. Hautes Études Sci. Publ. Math. No. 50 (1979), 171--202. 


\bibitem[Wa]{wa} Walters, Peter. \emph{An Introduction to Ergodic Theory.} Graduate Texts in Mathematics 79, Springer, 1981.

\bibitem[V1]{v1} Velozo, Anibal. \emph{Phase transitions for geodesic flows and the geometric potential.}  arXiv:1704.02562


\bibitem[V2]{v2} Velozo, Anibal. \emph{Thermodynamic formalism and the entropy at infinity of the geodesic flow}, arXiv:1711.06796. 






\end{thebibliography}
\end{document}